\documentclass[10pt, final]{amsart}
\usepackage[cp850]{inputenc}
\usepackage{amssymb}
\usepackage[mathscr]{eucal}
\usepackage{amscd}
\theoremstyle{plain}

\newtheorem{teor}{Theorem}
\newtheorem{prop}{Proposition}
\newtheorem{lemma}{Lemma}
\newtheorem{cor}{Corollary}

\theoremstyle{remark}
\newtheorem{remark}{Remark}
\newtheorem{example}{Example}

\title{Complete sets and completion of sets in Banach spaces}

\author{HORST MARTINI}

\author{PIER LUIGI PAPINI}

\author{MARGARITA SPIROVA}

\begin{document}
\maketitle {}

\textbf{ABSTRACT}
In this paper we study properties of complete sets and of completions of sets in Banach spaces. We consider the family of completions of a given set and its size; we also study in detail the relationships concerning diameters, radii, and centers. The results are illustrated by several examples.



\medskip

\section {Introduction}

The notion of diametrically maximal, or complete, set is around one century old. A set is \emph{complete} if all its proper supersets have a larger diameter.
A few decades ago the study of these sets, initially limited to finite dimensional spaces, was extended to Banach spaces of any dimension; one of the pioneering papers concerning this extension was \cite{BaP}. (Note that this paper is not so easily accessible and contains a few misprints.) New interest in these sets arose in the last few years. Among the recent papers on the subject are  \cite{MPP, Pv, Mor, MS, CP'}; see also the references contained in these papers.
More precisely, many recent papers deal with  questions like the following ones: Given a closed, bounded and convex set  $D$, which is the class of its completions? (Note that complete sets containing $D$ and having the same diameter as $D$ are called completions of $D$.) When does the class of complete sets contain it? When does   such a class consist of  a singleton? And in addition, are there good, or special completions?

Here we deal with these questions in a general  Banach space (which might be infinite dimensional). We single out and study some simple properties, trying to give a complete picture concerning the different situations which are possible.

\smallskip
In Sections 2 and 3, the sizes of completions for a given set $D$ are studied, with respect to the diameter and the radius of $D$.  Finally, in Section 4 we study some minimality conditions for sets concerning completions.

\bigskip

Let $X$ be a real Banach space (in finite dimensions also called normed or Minkowski space). We denote by $O$ the origin.  For $x \in X$ and $r \geq 0$,
$B(x,r)=\{y \in X: ||x-y|| \leq r\}$ denotes the \emph{ball} with center $x$ and radius $r$.

\smallskip
Let $D$ be a bounded, closed and convex set containing at least two points. In the following we shall always consider sets satisfying this condition. By $\delta(D)$ we will denote the \emph{diameter} of $D$, and  by $\partial(D)$ its \emph{boundary}.

\smallskip
We say that $D$ is $complete$ or $diametrically$ $maximal$  - (DM) for short - if \,$\delta(D \cup \{x\})>\delta(D)$ for every $x\notin D$.

A $completion$ of  $D$ is a (DM) set $D_m$ containing $D$ and such that $\delta(D)=\delta(D_m)$.

\smallskip
We shall also use the following notations (see \cite{BaP}):

\smallskip
$D'= \bigcap _{x\in D} \{B(x, \delta(D))\}$, called the \emph{ball intersection of} $D$, and

\smallskip
$D^c= \bigcap _{x\in X} \{B(x, \delta(D)): D \subset B(x, \delta(D))\}$, called the \emph{ball hull of} $D$.

\medskip
We have always $D \subset D^c \subset D'$. Moreover (see \cite[Propositions 3.1 and 3.2]{BaP}), we have  $\delta(D^c)=\delta(D); \;  (D^c)'=D'; \; D$  is complete $\Leftrightarrow D=D'$. Also we have $(D^c)^c=D^c$ (a misprint occurs in \cite [ Proposition 3.2] {BaP}).

\smallskip
We recall (see, for example,  \cite [p. 311] {Bav}) the known property

\smallskip
\noindent  ($P_1)  \quad  D^c= \bigcap _{x\in D'} \{B(x, \delta(D))\}$.

\medskip
We always have (see, for example, \cite [Proposition 2]{Mor}; or also \cite[Theorem 3] {PS}):

\smallskip
\noindent $(1) \quad D^c= \bigcap \, \{A: \;A \; \mathrm{is \;a \:completion \;of} \; D\};$

\smallskip
\noindent $(2) \quad D'= \bigcup \, \{A: \;A \; \mathrm{is \;a \;completion \;of} \; D \}$.

\smallskip

Thus $D$ has a unique completion if and only if its completion is $D^c=D'$; this is equivalent to the following equalities (see \cite[Theorem 3.7] {BaP} or  \cite [Corollary 3] {Mor}):

\smallskip
\noindent $(2') \quad  D^c$ is $(DM) \Longleftrightarrow  D'$ is $(DM) \Longleftrightarrow \delta(D)=\delta(D').$

\smallskip
We recall that for normed planes and spaces these sets have been considered in many papers: from the not so recent paper \cite {Bav}, to the recent ones \cite {Pv}, \cite{MSpi}, \cite{MRS}, and \cite{Mor}, whose results partly overlap with some results in \cite{BaP}.

\bigskip

\section {On the completions of a set and their sizes}

\medskip
For $x \in X$ set

\smallskip
$r(D,x)= \sup\{||x-d||:  d \in D\}$;

\smallskip
$r(D)= \inf\{r(D,x): x\in X\}$ \; ($radius$ of $D$);

\smallskip
$r(D,D)= \inf\{r(D,x): x\in D\}$   \; ($self$-$radius$ of $D$).

\smallskip
A point $c \in X$ such that $r(D,c)=r(D)$ is called a $center$ of D.
Note that  not always a center exists, but in finite dimensional case its existence is guaranteed. We have  always  $r(D) \leq r(D,D) \leq \delta(D) \leq 2 r(D)$, and  also, $r(D,D)=r(D)$
if $D$ is complete (see \cite [Theorem 3.3]{BaP}).

\medskip
Clearly, we always have

\smallskip
\noindent $(3) \quad \delta(D') \leq 2\,  r(D') \leq   2 \,
 r(D', D') \leq 2\, \inf \{r(D',x): x\in D\}
 \leq 2 \, \delta(D) \leq 4 \, r(D).$

\smallskip
 Note that    $ r(D', D') < \, \inf \{r(D',x): x\in D\}$ in Example 5 (in Section 3).    The following examples show that also (3) is sharp, since the equality $\delta(D') = 4 \, r(D)$ is possible.

\smallskip
\begin{example}
 In fact, our first example consists of three examples; the third one will be used further on.

\smallskip
A) Consider in $R^2$, with the max norm, the set $D=\{(x,0): \, 0 \leq x \leq 1 \}$.
We have:  $D=D^c$ and $ D'= \{(x,0): \, 0 \leq x \leq 1, \,  |y| \leq 1 \}$, and  so
$\delta(D') = 2$ and $r(D)=1/2$.

\smallskip
B) Consider the space $R^3$ with the sum (or  $\ell_1$) norm. Let $D$ be the convex hull of  $\{(1,1,0); \; (1,0,1); \; (0,1,1)\}$. We have: $\delta(D)=2$ and $r(D)=1  \;((1,1,1)$ is the unique center of $D$). Both points (0,0,0) and $(\frac{4}{3}, \frac{4}{3}, \frac{4}{3})$ belong to $D'$, and their distance is $4$. Also, according to (2'), $D$ has different completions.

\smallskip
C) $D$ is the set indicated in A), but in the space $R^2$ with the Euclidean norm. We have:

$D^c=B\big( (1/2, -\sqrt 3/2), 1\big) \cap B\big( (1/2, \sqrt3/2), 1 \big); \;  D'=B (O, 1 ) \cap B \big( (1,0), 1 \big)$.

\end{example}

For $d \in D$, set

\smallskip
$r'(D,d)= inf\{||x-d||:  x \notin D\}$;

\smallskip
$r'(D)= sup\{r'(D,x): x\in D\}$ \; ($inner$ $radius$ of $D$).

\smallskip
Note that if $D$ has empty interior, then $r'(D)=0$.

\medskip

Given $D, \; x\in D'$ means\,  $||x-d|| \leq \delta(D)$ for all \, $d \in D$. So we have that

\smallskip
\noindent $(P_2)$ \quad    if $x \in D'$, then $r(D,x) \leq \delta(D)$ \, (and conversely).

\medskip
 If \, $D$ has interior points and \, $B(d_o, \alpha) \subset D \;\, (d_0 \in D; \, \alpha>0)$, then $||x-d_o|| \leq \delta(D)-\alpha$ for all $x \in D'$;  so  $||x-y|| \leq  ||x-d_0||+ ||d_0-y|| \leq 2(\delta(D)-\alpha)$  for $x, y \in D'$. Therefore (3) can be improved to

\medskip
\noindent $(3') \quad \delta(D') \leq 2 \, \delta(D)-2 \, r'(D). $

\medskip
\noindent Moreover (with the same notations as above),
$x \in D'\smallsetminus D$ implies   \, $d'_o=d_o+\alpha \frac{x-d_o}{||x-d_o||} \in D$ and
$||x-d'_o||\leq r(D,x)-2 \alpha$. Thus,  denoting by  $H(A,B)$ the Hausdorff distance between the sets $A$ and $B$, we have

\medskip
\noindent $(3'')\quad H(D,D') \leq \delta(D)-2r'(D) \leq 2(r(D)-r'(D)) $.

\medskip
The inequality $(3')$ is sharp; see Example 1 A). Or also, use as $D$ the set denoted by $D'$ in the same example  (for which  $ r'=1/2$ and $\delta =2$). The same holds  for the inequalities in  $(3'')$. Moreover, no better estimate is possible if we consider $D^c$ instead of $D'$: we give an example where $D$ has nonempty interior, its completion is unique and equalities  hold in $(3'')$.

\begin{example}
Consider in  $R^2$, with the max norm, the set
$D=\{(x,y): 0 \leq x \leq 1, \, 0 \leq y \leq x\}$. We have:
$\delta(D)=1$, $r(D)=1/2$, $r'(D)=1/4$, $D'=D^c=\{(x,y): 0 \leq x \leq 1$,
$0 \leq y \leq 1\}$ (this is the unique completion of $D$), and $ H(D,D^c)=H(D,D')=1/2$.

 \end{example}

\smallskip
The inequality $H(D,D^c) \leq r(D)$ is not true in general (see Example 5 in Section 3), so $r(D)$ is not so useful in this context.  Now we shall consider $r(D,D)$.

\medskip
The following was proved in \cite [Theorem 3.5] {BaP}:

\smallskip
\noindent ($P_3$) \;  Let $D\subset B(x,r)$ (for some  $x \in X$ and $r \in R$). Then  $D$ has a completion contained in $B(x,r)$

\quad \; if \, $r \leq \delta(D)  \, $ (so $D^c \subset B(x, r) )$.

\smallskip
Indeed,  the last bound for $r$ was not indicated in  \cite {BaP}, but the proof given there only works for  $r \leq \delta(D)$. Example 4  A) below shows that such bound is crucial.

\smallskip
\begin{prop}
We always have

$H(D, D^c) \leq r(D,D).$
\end{prop}

\begin{proof}
If $r(D,D)=\delta(D)$,  the result is contained in $(3'')$.

Otherwise, let $\varepsilon \in \big(0,\,  \delta(D) - r(D, D) \big];$ take $d \in D$ such that $D \subset B \big(d, \,  r(D,D)+ \varepsilon \big) \subset B \big(d, \delta(D) \big)$. Then, according to $(P_3)$,  also $D^c \subset  B \big(d, r(D,D)+ \varepsilon \big) $. Thus $x \in D^c \implies ||x-d|| \leq r(D,D)+ \varepsilon$ with $d \in D$, and then $H(D, D^c) \leq r(D,D)+ \varepsilon$. Since $\varepsilon >0$ is arbitrary, we obtain the thesis.
\end{proof}
\smallskip
Note that the inequality $H(D, D') \leq r(D,D)$ is not true in general. Our next example shows this.

\smallskip
\begin {example}
Let, for $t \in [0, 1/2],  \, D_t=\{f \in C[-1,1]: 0 \leq f(x) \leq t(x+1) $ for $ -1 \leq x \leq 0;  \; 0 \leq f(x) \leq 1 $ for
 $ 0 \leq x \leq 1 \}$.
Among the completions, there are sets of functions satisfying, for $-1 \leq x \leq 0$,
 $f(x) \in [\alpha-1, \alpha$], $t \leq \alpha \leq 1$; or also   $t(x+1)-1 \leq f(x) \leq t(x+1)$.   We have: $r(D_t)=1/2$, $\delta(D_t)=1$, $r(D_t,D_t)=1-t$, $\delta(D'_t)=2$, and $H(D_t, D'_t)=1.$

\end {example}

\smallskip
The difference $\delta(D') - \delta(D) \;\, \big(\leq \delta(D)\big)$ measures, in a sense, how  different the completions of $D$ can be. We have

\smallskip
\begin{prop}

Let

\smallskip
$H(D)=sup \, \{H(D_1,D_2): D_1, D_2$ are completions of $D$\}.

\smallskip
Then \, $\delta(D') - \delta(D) \leq H(D)$.

\smallskip
Moreover, equality holds  if any complete set C satisfies the following,  slightly stronger condition (usually called "$constant$ $diameter$", see \cite[Section 3]{MPP}):

\smallskip
\noindent (CD) \quad  $r(C,x)=\delta(C)+ \mathrm{dist}(x,C)$  \,  for every $x \notin C$, where $\mathrm{dist}$ denotes the usual distance with respect to the norm under consideration.

\end{prop}

\begin{proof}
We give the proof assuming that the values involved are attained; otherwise they can be arbitrarily well approximated,  and the proof can be easily adapted (so the result is still true).
Let $\delta(D') = \delta(D) + \alpha, \; \alpha >0$ (if $\alpha=0$, then $D'$ is the unique completion of $D$,   and there is nothing to prove). Take $x, y \in D'$ such that $||x-y||=\delta(D')$; according to (2) there are two completions of $D$, say $D_1$ and $ D_2$, such that $x \in D_1, \, y \in D_2$. Since  $r(D_1,x) \leq \delta(D)$  (see ($P_2$)), to reach $y$ from $D_1$ we have to enlarge $D_1$ at least by $\alpha$. So the conclusion for the first part holds.

To prove the second part, let  $H(D)=H(D_1,D_2)   \; (D_1 $ and $D_2$ again  being completions of $D$). Assume,  for example  that for every $\varepsilon>0$, there exists $x \in D_2 \subset D'$ such that $||x-y||> H(D)-\varepsilon$ for every $y \in D_1$  (otherwise, the role of $D_1$ and $D_2$ should be exchanged). Then (by using (CD))  $r(D_1,x) \geq \delta(D_1)+ H(D) - \varepsilon$; \, so $\big(\delta(D_1)=\delta(D) $\big) there exists $ \bar{y}\in D_1 \subset D'$ such that $||x- \bar{y}||>\delta(D) + H(D) - 2  \varepsilon$. This  implies $\delta(D') \geq \delta(D)+ H(D)$.
\end{proof}

\smallskip
\begin{remark}
We do not know if the assumption given  for the second part of the previous proposition is necessary.

\smallskip
We always have   $H(D) \leq \sup \, \{H(D_1,D): D_1$ is a completion of $D$\}. (Note that by the above Example 1 C), strict inequality is possible.)

\smallskip
\noindent In fact, let $\varepsilon>0$. According to Proposition 2 there are  $D_1, D_2$ such that \, $\sup \,  \{\mathrm{dist} (x, D_2): x \in D_1\} > H(D)- \varepsilon$. Then  \, $\sup \, \{\mathrm{dist} (x, D):  x \in D_1\} >H(D) - \varepsilon$, so the conclusion.

\smallskip
Also, we  have that    \,  $\inf \,\{H(D_1,D): D_1$ is a completion of $D\} \geq H(D, D^c)$.

\noindent In general this is not an equality (see the same example quoted above).

\end{remark}

\medskip

\begin{teor} Let $D_1$ and $D_2$ be different complete sets. We have:

\smallskip
(a) if  there is no inclusion between $D_1$ and $D_2$, then $D_1 \cup D_2$ is not complete;

\smallskip
(b) if  \, $\delta(D_1)=\delta(D_2)=\delta (D_1 \cap D_2)$, then $D_1 \cap D_2$ is not complete.
\end{teor}

\begin{proof} (a): Let $\delta(D_1)=d_1 \leq d_2 = \delta(D_2).$  We cannot have $\delta (D_1 \cup D_2) \leq d_2$, since $ (D_1 \cup D_2)$ strictly contains $D_i, \, i=1,2$. Let $\delta (D_1 \cup D_2)=d_2+\varepsilon, \; \varepsilon>0.$

\medskip

Case 1:    Let $ D_1 \cap D_2 \neq \emptyset$.

Assume that  $D_1 \cup D_2$ is complete.  There is a point $ z$ which is at the same time in $ \partial ( D_1 \cup D_2)$ and in
 $ \partial ( D_1 \cap D_2)$. This implies  \,  $\sup\{ ||z-y||: y \in  D_1 \cup D_2\} \leq d_2$; but also (boundary points are endpoints of diameters in complete sets) sup$\{ ||z-y||: y \in  D_1 \cup D_2\} = d_2 + \varepsilon $. This contradiction shows that
  $D_1 \cup D_2 $  is not complete.

\medskip
Case 2:   $D_1 \cap D_2= \emptyset$.

\smallskip
In this case   $ D_1 \cup D_2 $ is not convex. In fact, let  $x \in D_1, \,  y \in D_2$.  The sets $[x, y] \cap D_1$ and  $[x, y] \cap D_2$ are convex (and disjoint). Let    $[x, y] \cap D_1=[x, x'], \,  [x, y] \cap D_2=[y', y]$  ($x' \neq y')$.  So $(x', y')$ is not contained in $ D_1 \cup D_2$.

\bigskip
(b): It is enough to observe that  $D_1 \cap D_2$ is strictly contained in each  of the two sets $D_i, \,  i=1,2$.
 \end{proof}

\begin{cor} If $D$ has two different completions $D_1$ and $D_2$, then neither $D_1 \cup D_2$ nor $D_1 \cap D_2$ can be complete.
 \end{cor}

\begin{proof} The first part follows from part (a) of the previous theorem. For the second part, note that \, $D_1 \cap D_2$  contains $D$ (so $\delta(D_1 \cap D_2)= \delta(D)$), and  then apply part (b) of the previous theorem.
\end{proof}

\bigskip

To prove our next result, we need the following lemma.

\begin{lemma} Let $D$ be a set, and let  $\alpha, \varepsilon >0$. Set

\smallskip
$A= \bigcap_ {x\in D}  B(x, \alpha);  \;    A_ {\varepsilon} =  \bigcap_ {x\in D}  B(x, \alpha + \varepsilon)$.

\medskip
If $A \neq \emptyset$, then $ \delta( A_ {\varepsilon} ) \geq \delta(A)+2 \varepsilon$.

\end{lemma}

\begin{proof} Let $\delta(A)=d \;  (d > 0\; \mathrm{by}\; A \neq \emptyset)$. Given $\sigma>0$, take $y_1$ and $y_2$ in $A$ such that   $||y_1-y_2|| \geq d-\sigma$. On the line joining $y_1$ and $y_2$, take $y_1'$ and $y_2'$ so that  $  ||y_1-y_1'||=\varepsilon= ||y_2-y_2'||, \;  ||y_1'-y_2'||=   ||y_1-y_2||+2 \varepsilon$.  Since $y_1 \in  A$, we have $y_1' \in A_{\varepsilon}$; similarly,  $y_2' \in A_{\varepsilon}$, so $\delta(A_{\varepsilon}) \geq d-\sigma+2 \, \varepsilon$. Since $\sigma>0 $ is arbitrary, then the thesis follows.
\end{proof}

\smallskip
\begin{remark}
 In the above lemma, we can have  $ \delta( A_ {\varepsilon} ) > \delta(A)+2 \varepsilon$.  For example, if we consider as $A$ the set denoted by $D'$ in Example 1 C)  and we take $\alpha=\varepsilon=1,$ then $ \delta(A)= \sqrt 3,\;   \delta( A_ {\varepsilon})= 3 \sqrt3$ .
\end{remark}

\medskip

\begin{teor}
Let $A, B$ be sets such that  $\delta(A) \neq \delta(B)$.  Then  $A^c \neq B^c$ and $ A' \neq B'$.

\end{teor}

\begin{proof}
If $\delta(A) \neq \delta(B)$, then $\delta(A^c) \neq \delta(B^c)$, so  $A^c \neq B^c$.

\smallskip
For the second part,  assume  (by contradiction) that   $A'=B'$. For example,  let  $\delta(A)=d, \; \delta(B)=d+\varepsilon, \;  \varepsilon \geq 0$;  then (see ($P_1$))

\smallskip

$A^c= \bigcap  \{B(x, d)): {x\in A'}  \}$

\noindent and

$B^c= \bigcap  \{B(x, d+\varepsilon): {x\in A'}  \}$.

\smallskip
Therefore, according to Lemma 1, we have   $ d+ \varepsilon = \delta( B) = \delta( B^c) \geq \delta( A^c) + 2\varepsilon $. This proves that $\varepsilon=0$, so  \, $\delta( B) = \delta( B^c) =\delta( A^c) = \delta( A)$.
\end{proof}

\medskip
It is clear that, in general, the inclusion $A\subset B$ together with the equality $\delta(A)=\delta(B)$ implies

\smallskip
$A^c\subset B^c; \; B'\subset A'$.

\smallskip Moreover, the inequality $\delta(B')<\delta(A')$ is possible: let $A$ be the set $D$ in Example 1 C) and $B=B(O, 1)$.

\medskip
Under the same assumptions on $A$ and $B$,  if $A$ has a unique completion (this means $A^c=A'$), then $B$ has the same unique completion since $A^c \subset B^c \subset B' \subset A'$. So we have

\medskip
\noindent $(P_4)$ If $  \; A\subset B, \; \delta(A)=\delta(B)$, and $  A$ has a  unique completion $A'$, then $A'$ is also the unique

 \quad completion of $B$.

\medskip
But it is possible that $B$ has a unique completion (in particular, $B$ is complete), and $A \subset B$ has more completions.

\smallskip
Also we can have  $\delta(A)<\delta(B)$ and   $\delta(A')=\delta(B')$.

\smallskip
In general, the inclusion $A\subset B$
does not imply that  any inclusion between $A^c$ and $B^c$ or between $A'$ and $B'$ holds  (unless $A$ and $B$ have a unique completion). We give an example showing this.

\smallskip
\begin{example}
Let $X$ be the Euclidean plane.

\smallskip
A) Let  $A$ be the equilateral triangle determined by the points  $(-1, 0),  (1, 0),  (0, \sqrt3) $,  and  $ B$ be  the ball $B\big((0, a), \sqrt  {a^2+1 }\, \big), \, a \geq 1/ \sqrt3$. The boundary of $B$ contains the first two vertices of $A$, and we have:  $A \subset B$, $\delta(A)<\delta(B)$,  and $B^c=B $ is not contained in $A^c$. Moreover, for $a> \sqrt3$ the unique completion of $A$ (a Reuleaux triangle $T=A^c=A'$) is not contained in  $B=B^c=B'$.

\smallskip
B) Let  $A=\{(x,0): \, |x|\leq 1/2 \}$  and $B=B(0, 1/2+ \varepsilon)$ with $\varepsilon>0$ "small". Then there is no inclusion between $B'=B$ \, and  $A'$.

\end{example}

\smallskip We prove another result. For two different sets $A$ and $B$ it is possible to have $A^c=A'=B'=B^c$ (for example, let $B$ be  the unique completion of an incomplete set $A$). In any case, the following fact is true.

\begin{prop}
Given two sets  $A $ and $ B$,  we have

\smallskip
$A^c = B^c  \Leftrightarrow A'=B'$.

 \end{prop}

\begin{proof}
Let $A^c=B^c$. Then $A'=(A^c)'=(B^c)'=B'$.

\smallskip
Conversely, let $A'=B'$. Then, according to Theorem 2, we have  $\delta(A)=\delta(B)=\delta$. So, by using  $(P_1)$, we obtain:  $A^c= \bigcap_ {x\in A'}  B(x, \delta) =  \bigcap_ {x\in B'}  B(x, \delta) =B^c $.
 \end{proof}

\medskip

\section {Completion of sets, radii,  and centers}

\medskip
Some of the results in this section have been indicated, for finite dimensional spaces,  in  \cite{MMS}.

\medskip
 Every center of $D$ is contained in $D'$ (see  \cite [Theorem 3.3]{BaP}),  but not necessarily in $D^c$, also when there is a unique center for $D$. In fact: for the set $D$ in Example 1 B),  $(1,1,1) \notin B(O, \, 2)$, where $2=\delta(D)$, and  $D \subset B(O, \, 2)$;  so also $D^c \subset   B(O, \, 2)$.

\medskip
It is clear that in general $D^c$  does not contain a completion of $D$.

\medskip
Let $D \subset B(x, \delta(D))=B_x$. If $x \in D$, then $D' \subset B_x$, and  so all completions of $D$ are contained in $B_x$. But otherwise, in general, this is not true. For example, consider the set  $D$ in Example 1 C) \, ($\delta(D)=1$) and  $x=(1/2, 1/2)$.

\medskip

According to $(P_3$),  $r(D)=r(D^c)$ (see \cite [Corollary 3.6]{BaP}) and

\smallskip

  $\;\; r(D)=\inf \{r(A): A$ is a completion of $D$\}.

\smallskip
Clearly,   $r(D)$ is also a minimum in the above formula when $D$ is complete, but also in the following cases:

\smallskip
 - $D$ has a unique completion $D^c$;

\smallskip
 - $D$ has a center (see the discussion above).

\smallskip
But in some cases the infimum is not a minimum; see  \cite [Example 3.3] {BaP}. Also (and as already said) there exist complete sets without centers (see \cite [Example 1] {CP}).

\medskip
Clearly, in general $r(D') \geq r(D)$ (inequality is strict in Example 1 C)). According to (3), $r(D')\leq 2 r(D)$, and this estimate is sharp (see Example 1 A)).

\smallskip
Note that a set can have a unique center and different completions (see Example 1 C)).

\smallskip
In general, different completions of a given set can have different radii; also, we can have different completions with the same radius. The range of  \, $r(D_m)$, $D_m$ denoting completions of $D$, can be the whole interval $[r(D), 2 r(D)]$. To see this, look at Example 3, with $t=0$. In fact, we have:  $D$ has completions with the same radius $r=r(D)$ \, ($\alpha=1)$ and completions with radius $r>r(D)$ ($r=1$ for $\alpha=0$).

The same example, with $\alpha=0$, also shows that  complete sets can have more centers.

\medskip
Our next result extends \cite [Corollary 3]  {MMS}.

\begin{teor}
The sets $D$ and $D^c$ have the same centers and the same completions.

\end{teor}

\begin{proof}
Recall that  $D$ and $D^c$ have the same diameter and the  same radius.

Let $r(D^c)=r(D)=r$. If $c$ is a center of $D^c \;  (\mathrm{i.e.},\; D^c \subset B(c,r))$, then it is also a center of $D$. Conversely, if  $c$ is a center of $D \;  (\mathrm{i.e.},\; D \subset B(c,r))$, then by $(P_3)$ $D^c \subset B(c,r)$, and so $c$ is also a center of $D^c$.

For the second part: a completion of $D$ must contain $D^c$ (see (1)), so it is also a completion of $D^c$. The converse is clear.
\end{proof}

Now we compare the completions of $D'$ with those of $D$.

\smallskip
If  $D$ has more completions, then $\delta(D')>\delta(D)$, so   $D$ and $D'$ have different completions. Moreover (see \cite [Theorem 3.8]{BaP}),  we have $ \delta\big((D')'\big)>\delta(D'))$, so also $D'$ has more completions.

\smallskip
 Let $D$ have a unique completion $D'$. Then $r(D)=r(D')=r(D',D')$. But we can have $r(D,D)>r(D)$; see the next example. Note that if, moreover, $D$ has no center, then (since $r(D)=r(D'))$ the same is true for all sets $A$ satisfying  $D \subset A \subset D'$.

 \smallskip
\begin {example}
Let $D=\{f \in C[0,1]: f(0)=0; \, 0 \leq f(x) \leq 1\}$.  Its unique completion is $D'=B(\bar{f}, 1/2)$, where $\bar{f}$ is the constant function 1/2 (we have $r(D)=r(D')=r(D', D')=1/2; \;\, r(D,D)=1=\delta(D)=\delta(D')$).
\end {example}

We noticed that, in case of a unique completion, $r(D)=r(D^c)=r(D')$. The converse is
 not true. Namely, the condition $r(D)=r(D')$ does not imply $\delta(D)=\delta(D')$, as our next example shows (see \cite [Theorem 3.7, $(k)\nRightarrow (h)$]{BaP}).

\medskip
\begin {example} Consider the space  $c_o$ of all sequences converging to 0, with the max norm. Let $D=\{x=(x_1, x_2, ... ,x_n, ...) :\,  x_1=0; \; 0 \leq x_i \leq 1$ for$ \; i \geq 2\} \; (\delta(D)=1)$. The set $D$ has different completions ($\delta(D')=2)$; but we have  $r(D)=r(D')=r(D,D)=1.$

 \end{example}

\medskip
Let $A \subset B$ and $\delta(A)=\delta(B)$. Then $A^c \subset B^c$,  and so $r(A^c)\leq r(B^c)$. In general, the last inequality is not  an equality. For example, consider  in the Euclidean plane  $A=\{(x,0): \, |x|\leq 1/2 \}$ and $B$ an equilateral triangle based on $A$.

\noindent Also, under the same assumptions  we have  $A' \supset B'$. So $r(A')\geq r(B')$, and in general the last inequality is not  an equality, as the same example above shows.

\smallskip
If $A\subset B$ and $\delta(A) < \delta (B)$, then,  concerning $r(A')$ and $r(B')$,  all cases are possible ($<, \, =, \, >$). In fact, the possibility $r(A') < r(B')$ is trivial; equality  is possible according to Example 6 above (take $A=D, \, B=D'$). Concerning $r(A') > r(B')$, see Example 4 B).

\bigskip
We present a result indicated in \cite{BaP} (see Remark to Corollary 3.4 there) and an immediate consequence of it. We limit ourselves to the consideration of  sets where centers exist, but the general case could be treated in a similar way, giving estimates for the set of approximate centers of $D$.

\smallskip
For a set $D$, we denote by  $C_D$ the set of its centers.

\begin{prop}

For every set $D$ the following inequality is true:

\smallskip
$\delta(C_D) \leq 2\,r(D)-2\,\delta(D)+\delta(D')$.

\smallskip
In particular, if $D$ has a unique completion, then

\smallskip
$\delta(C_D) \leq 2\,r(D)-\delta(D)  \leq \delta(D)$.
\end{prop}

\begin{remark}
 The inequality $\delta(C_D) \leq 2\,r(D)-\delta(D)$ is not true in general. In fact, in Example 6 we have $\delta(D)=r(D)=1$, $\delta(C_D)=\delta(D')=2$.
 The first inequality is sharp; see Example 1 A).

\noindent  Concerning the second statement in Proposition 4, the first  inequality  is sharp; for example, it becomes an equality for a ball. Of course,  also the second is sharp; see Example 6.

The second part also implies the following fact: if $\delta(C_D)>\delta(D)$, then $D$ has more completions.
\end{remark}

\medskip

\section {Completions and Minimality}

The sets we are considering will always be assumed to be bounded, closed and convex with diameter $>0$.

\smallskip
Let $C$ be a complete set of diameter $d$.

Consider a set $D \subset C$ with   $\delta(D)=d$ (in particular,  $D=C$). Then set

\smallskip
$T(D,C)=\{A \subset D: \; C $ is a completion of $A \}$.

\smallskip
\noindent Also, let $C$  be  the unique completion of $D   \; (D^c=D'=C)$; set

\smallskip
$U(D,C)=\{A\subset D: \; C $ is the unique completion of $A \} = \{A\subset D: \, A^c=A'=C \} $.

\smallskip
We shall discuss minimality for completions, or for unique completions.

\smallskip
If \,  T(D,C) (resp. U(D,C)) is  the singleton $\{D \}$, then we say that $D$ is (mC) (resp.:  (muC)). Otherwise, if $D$ is  not (mC), or not  (muC), then  an application of Zorn's lemma (every chain has a lower bound) shows that $D$ contains minimal subsets with that property.   In other words,  there are subsets of $D$ which are minimal  in the sense of inclusion, that still have $C$ as completion (resp.: as unique completion).

\smallskip
Note that ''$D$ is (mC)`` means the following: $D \subset C, \;  \delta(D)=d, \;  \delta(S)< d$ for every proper subset $S$ of $D$.
So the condition ''$D$ is (mC)`` is equivalent to:  the maximal width of $S$ is smaller than $d$  for every proper subset $S$ of $D$ (see \cite [Proposition 4]{CP'}). Apparently, a related notion (less tractable) is the following:  the minimal width of a convex body $D$ is smaller for every proper subset $S$ of $D$ (i.e., $D$ is \emph{reduced}; see, for example, \cite  {LM}).  Reducedness does not imply (mC) (look at the equilateral triangle in the Euclidean plane);  a segment is (mC) but it is not reduced (it is not a body).

\medskip
We shall discuss the following questions.

\smallskip
(Q1) \, Which sets  $D$ are  (mC)?

(Q2) \, Which sets $D$ are (muC)?

\smallskip
The answer to (Q1), at least in finite dimensional spaces, is trivial. The answer to (Q2) seems to be difficult.

\smallskip
\begin {prop} If $X$ is a finite dimensional normed  space, then, given a complete set $C$ of diameter $d$, $D \subset C$ is {\rm (mC)} if and only if it is a segment of diameter (length) $d$ contained in $C$.

\end {prop}

\begin {proof}   If   $D$ is a segment of diameter (length) $d$, then any proper subset of $D$ has diameter $<d$, so $C$ is not a completion of it.

\smallskip Assume that $D$  is (mC)  \,  (i.e., $\delta(D)=d$). Take a diametral pair $x, y$ of $D$. Then the segment $[x,y]$ is (mC). If  $D$ is not a segment, then  $[x,y]$ is a proper subset of $D$, so $D$  cannot be (mC).
\end {proof}

\smallskip
\begin {remark}
 Proposition 5 is not true if $X$ is infinite dimensional. In fact, in this case we can consider a set $D$ without a diametral pair (see \cite  {Ve}); this means that if $\delta(D)=d$, then for every pair $x, y \in D$   we have $||x-y||<d$. Now let $C$ be a completion of $D$.  There is a minimal subset $A$ of $D$ having $C$ as a completion. But $A$ cannot be a segment, since every segment  contained in $D$ has diameter (length) $<d$.

\end {remark}

\bigskip
Note that a segment can have a unique completion or more completions, as the next example shows.

\begin {example}
Let $X$ be the plane with the max norm.  Let $C=\{(x,y): \,  0 \leq x \leq 1;  \, 0 \leq y \leq 1 \}, \,  D_1=\{(x,y):  \, 0 \leq x \leq 1; \, y=0 \}, \,   D_2=\{(x,y): \,  0 \leq x \leq 1; \, x=y \}, \,  D_3=\{(x,y): \,  0 \leq x \leq 1; \,  y=1-x  \}.$  Then $D_1$ has more completions, while $D_2$ and $D_3$ have $C$ as unique completion.

\end {example}

\smallskip
 Now we discuss ($Q_2$). If $D$ has the unique completion $C$, then $\delta(D)=d= \delta(C)$, and moreover:  $D^c=D'=C$.  This means that (see $(P_1)$):

\smallskip
$D^c= \bigcap_ {x\in C}  B(x, d) =   \bigcap_ {x\in D}  B(x, d)=D',$

\smallskip
while if $S$ is a proper subset of $D$, of diameter $d$, then $\delta(S') = \delta \big(\bigcap_ {x\in S}  B(x, d ) \big) >d. $

\smallskip
Note that in this case, since $D^c=D'=C$, then $S^c$ should be strictly contained in $C$   (according to  Proposition 3).

\medskip
Segments are (muC) only if they have a unique completion $C$ (see Example 7).  We now give an example of a (muC) set which is not a segment.

\begin {example} Let $X$ be the plane with the max norm.  Let $C=\{(x,y): \,  0 \leq x \leq 1;\,  0 \leq y \leq 1 \}, \, \,   D_4=\{(x,y):  0 \leq x \leq 1; \,  x/2 \leq y \leq 1-x/2  \}$. Then $  D_4$ has $C$ as unique completion,  while every proper subset of $D_4$ has more completions.  The same is true if we consider an equilateral triangle in the Euclidean norm.

\end {example}

\smallskip
Minimal elements are not unique in general.  In Example 7, both $D_2$ and $D_3$ are (muC) (for $C$ as defined there); in that example,  $C$ is the unique set with diameter 1, containing both $D_2$ and $D_3$.

\smallskip
Note that if $M_1$ and $M_2$ are two different sets being (muC), then it is clear that $\delta(M_1 \cap M_2)<\delta(C)$.

\smallskip
In Example 8, $D_4$ is not complete (its unique completion is $C$);  each segment in the boundary of  $D_4$ is (mC).

\smallskip

\normalsize

\medskip

\small
Authors' addresses:

\smallskip
Horst Martini,   Fakult\"{a}t f\"{u}r Mathematik, TU Chemnitz, D-09107 Chemnitz, Germany. e-mail:  horst.martini@mathematik.tu-chemnitz.de

Pier Luigi Papini, Via Martucci 19, 40136 Bologna, Italy. e-mail: plpapini@libero.it

Margarita Spirova, Fakult\"{a}t f\"{u}r Mathematik, TU Chemnitz, D-09107 Chemnitz, Germany. e-mail: margarita.spirova@mathematik.tu-chemnitz.de

\bigskip
\bigskip
\end  {document}